\documentclass[11pt]{amsart}
\usepackage{amssymb, amsmath, amsthm, amsfonts}
\usepackage{epsfig}
\usepackage{latexsym,amssymb}
\usepackage{MnSymbol}
\usepackage{a4wide}
\usepackage{times}
\newtheorem{theorem}{Theorem}[section]
\newtheorem{lema}{Lemma}[section]
\newtheorem{prop}{Proposition}[section]
\newtheorem{cor}{Corollary}[section]

\newtheorem{remark}{Remark}[section]

\def\abs#1{\left| #1\right|}

\def\s{\mathbb{S}}
\def\sfe{\mathbb{S}^{n-1}}
\def\R{\mathbb{R}}

\def\N{\mathbb{N}}

\def\K{\mathcal{K}}

\def\c{C}

\def\Sf{\mathfrak{S}}
\def\Mat{{\mathcal S}}
\def\F{\mathcal F}

\def\Sa{\mathrm{S}_{n-1}}
\def\Ss{\mathrm{S}_{2}}
\def\Saplane{\mathrm{S}_{1}}

\def\H{\mathcal{H}}

\newcommand{\meares}{\lefthalfcup}

\begin{document}
\title[Characterization of some mixed volumes]{A characterization of some mixed volumes\\
via the Brunn--Minkowski inequality}
\author[Andrea Colesanti, Daniel Hug, and Eugenia Saor\'{\i}n G\'omez]{Andrea Colesanti,
Daniel Hug, and Eugenia Saor\'{\i}n G\'omez}
\address{Dipartimento di Matematica ``U. Dini",
Viale Morgagni 67/A,
50134 Firenze, Italy}
 \email{colesant@math.unifi.it}
\address{Karlsruhe Institute of Technology (KIT), Department of Mathematics, 
D-76128 Karlsruhe, Germany}
\email{daniel.hug@kit.edu}
\address{Fakult\"at f\"ur Mathematik, Otto-von-Guericke Universit\"at Magdeburg, \newline
Universit\"atsplatz 2, D-39106 Magdeburg, Germany}
\email{eugenia.saorin@ovgu.de}
\subjclass[2010]{Primary: 52A20, 52A39; secondary: 52A40, 26D10}
\keywords{Convex body, Brunn--Minkowski theory,
Minkowski inequality, valuation, mixed volume, area measure,  variational calculus}
\date{\today}
\begin{abstract} We consider a functional $\mathcal F$ on the space of
convex bodies in $\R^n$ of the form
\begin{equation*}
{\mathcal F}(K)=\int_{\sfe} f(u)\,\Sa(K,du)\,,
\end{equation*}
where $f\in\c(\sfe)$ is a given continuous function on the unit sphere of $\R^n$, 
$K$ is a convex body in $\R^n$, $n\ge 3$, and $\Sa(K,\cdot)$ is the area
measure of $K$. We prove that $\F$ satisfies an inequality of
Brunn--Minkowski type if and only if $f$ is the support function of a convex body,
i.e., $\F$ is a mixed volume. As a consequence, we obtain a characterization
of translation invariant, continuous valuations which are homogeneous
of degree $n-1$ and satisfy a Brunn--Minkowski type inequality.
\end{abstract}
\maketitle

\noindent

\section{Introduction}

In this paper, we consider functionals $\mathcal F \ : \ \K^n\to\R$ on the space $\K^n$ of convex bodies
in Euclidean space $\R^n$, $n\ge 2$, of the form
\begin{equation}\label{0.1}
{\mathcal F}(K)=\int_{\sfe} f(u)\,\Sa(K,du)\,,
\end{equation}
where $f\in\c(\sfe)$ is a (given) continuous function on the unit sphere $\sfe$ of $\R^n$,
$K$ is a {\it convex body} (a non-empty, compact, convex subset
of $\R^n$) and $\Sa(K,\cdot)$ is the {\it area measure} of $K$ (we refer
to the next section for definitions). The dependence of the functional $\mathcal{F}$ on 
the given function $f$ will be  clear from the context in the sequel.

Basic properties of  area
measures imply that such a functional is always translation invariant, continuous with
respect to the Hausdorff metric and homogeneous of degree $n-1$ with respect to
dilatations.  The latter means that
$$
{\mathcal F}(s\,K)=s^{n-1}{\mathcal
F}(K)\,,\quad \,K\,\in\K^n\,,\, s\ge0\,.
$$
Moreover, a functional  $\mathcal F$ defined via \eqref{0.1} is a {\it valuation}. The valuation
property requires that
$$
\mathcal F(K_0\cup K_1)+\mathcal F(K_0\cap K_1)= \mathcal F(K_0)+\mathcal
F(K_1)
$$
holds for all $K_0\,,\,K_1\in\K^n$ such that $K_0\cup K_1\in\K^n$. 

Conversely, a result of McMullen
(\cite{McMullen80}) states that every continuous, translation invariant
valuation, homogeneous of degree $n-1$, is of the form (\ref{0.1}). 
If $f$ is the support function of some fixed convex body $L$, then
$\mathcal F$ is a {\it mixed volume}. More precisely, according to
common notation in the theory of convex bodies, we have
$$
\mathcal F(K)=n V(K,K,\dots,K,L)=n V(K[n-1],L)\,,\quad\,K\in\K^n;
$$
for the definition of mixed
volumes we refer to \cite[Chapter 5]{Schneider}. In this case,
$\mathcal F$ is non-negative and satisfies the
following inequality of Brunn--Minkowski type (see \cite[Theorem
6.4.3]{Schneider}):
\begin{equation}\label{0.3}
\F\left((1-t)K_0+tK_1\right)^{1/(n-1)}\ge
(1-t)\F(K_0)^{1/(n-1)}+t\F(K_1)^{1/(n-1)}\,,
\end{equation}
for all $K_0\,,\,K_1\in\K^n$ and  $t\in[0,1]$ (where the set
addition is the usual Minkowski addition). The exponent
appearing in this inequality is the reciprocal of the order of homogeneity of
$\mathcal F$.
Inequality (\ref{0.3}) is a consequence of the Aleksandrov--Fenchel
inequalities, which are among the deepest results in Convex Geometry.
It belongs to the same family of inequalities as the classical
Brunn--Minkowski inequality, which states that the volume raised to the
power $1/n$ is a concave functional on $\K^n$. For further information on this topic,
we refer the reader to the survey paper
\cite{Gardner}, which is entirely devoted to the Brunn--Minkowski inequality and
its connections to various other  branches of
mathematics.

More generally, we say that a functional $\mathcal{G}\ : \ \K^n\to\R_+$, which is positively homogeneous of
degree $\alpha$ (for simplicity, assume $\alpha\ne0$), satisfies an inequality of Brunn--Minkowski type,
if $\mathcal G^{1/\alpha}$ is concave on $\K^n$, that is
$$
\mathcal{G}\left((1-t)K_0+tK_1\right)^{1/\alpha}\ge
(1-t)\mathcal{G}(K_0)^{1/\alpha}+t\mathcal{G}(K_1)^{1/\alpha}\,,
$$
for all $K_0\,,\,K_1\in\K^n$ and  $t\in[0,1]$.
Examples of
functionals sharing these properties arise in quite different contexts:
they include a large
number of geometric functionals, as well as important examples coming from
different areas, like the Calculus of Variations (see, for instance,
\cite{Colesanti2005}). Understanding whether there are general conditions
such that a given functional satisfies a Brunn--Minkowski type inequality
is a fascinating problem, but maybe too ambitious. On the other hand, as a
first step in this direction, one could try to answer the question in some
restricted class of functionals, which is what we do in this paper by
focusing on functionals of the form (\ref{0.1}).

In dimension $n=2$, the inequality (\ref{0.3}) becomes an equality, and in fact
this is true for any choice of the function $f\in C(\s^1)$ (irrespective
of whether it is a support function or not). Indeed, due to the relation
$$
\Saplane(K_1+K_2,\cdot)=
\Saplane(K_1,\cdot)+\Saplane(K_2,\cdot)\,,\quad\, K_1,K_2\in\K^2\,,
$$
in the Euclidean plane condition \eqref{0.3} is satisfied with equality for every $f\in C(\s^1)$.
Hence the problem of characterizing $f$ via inequality (\ref{0.3}) is reasonable
for $n\ge3$ only. In addition to (\ref{0.3}), we also consider the weaker condition
\begin{equation}\label{0.3b}
\F((1-t)K_0+tK_1)\ge \min\{\F(K_0), \F(K_1)\}\,,
\end{equation}
for all $K_0,K_1\in\K^n$ and $t\in[0,1]$. Condition \eqref{0.3b} has
 the advantage of not requiring any {\it a priori} assumption on the sign of
$\mathcal F$. Obviously, if $\mathcal F\ge0$ on $\K^n$ is such
that (\ref{0.3}) holds, then $\mathcal F$  also satisfies (\ref{0.3b}).

The paper is devoted to proving that, {\it for $n\ge 3$,
(\ref{0.3b}) characterizes mixed volumes among functionals of type \eqref{0.1}.}

\begin{theorem}\label{gral statement}
Let $n\ge 3$ and $f\in\c(\sfe)$. Then the functional $\F$ defined
as in (\ref{0.1}) satisfies (\ref{0.3b}) if and only if $f$ is the support function
of a convex body.
\end{theorem}

According to the result of McMullen mentioned before, Theorem \ref{gral
statement} can be rewritten in terms of valuations as follows.

\begin{theorem}\label{gral statement val}
Let $n\ge3$ and let $\mathcal{V}$ be a valuation on
$\K^n$. Then $\mathcal{V}$ is continuous, translation invariant, homogeneous of degree $n-1$ and satisfies
inequality (\ref{0.3b}) if and only if there exists a convex body $L\in\K^n$ such that
$$
\mathcal{V}(K)=V(K[n-1],L)
$$
for all $K\in\K^n$.
\end{theorem}

The proof of the `only if' part of Theorem \ref{gral statement} proceeds by induction over the
dimension. In the inductive procedure, the most difficult part is the initial step,
i.e.~the proof in the three-dimensional case, while the reduction to lower
dimensions, carried out in Section \ref{sV}, is much easier.

The proof of the three-dimensional case is presented in Section \ref{s: 3 dim case}.
Roughly speaking, we compute the second variation of the functional $\F$, as
a quadratic form on test functions. The Brunn--Minkowski inequality
(\ref{0.3b}) implies that this is a negative semi-definite functional on a certain
class of test functions. By a further specialization in the choice of the
test functions, we obtain that the Hessian matrix of the homogeneous
extension of order one of $f$ is positive semi-definite, i.e.~$f$ is a
support function. This argument was initially inspired by some ideas
contained in \cite{Colesanti2008} and \cite{Colesanti-Saorin}, where the sign of
the second variation of functionals satisfying inequalities of Brunn--Minkowski type
was used to derive functional inequalities of Poincar\'e type.

Even though the idea upon which
the proof is based is not too involved, to adapt it to the general situation
in which $f$ is just continuous, required several technical steps (contained in
Section \ref{dim3dss3}). For this reason,  we outline in Section \ref{dim3dss2}
the proof of Theorem \ref{gral statement} in the three-dimensional
case under the additional assumption that $f$ is sufficiently smooth, symmetric
and positive.  This should help the reader to identify the essence of the general argument.

\section{Preliminaries}\label{sII}

We work in the $n$-dimensional Euclidean space $\R^n$,
$n\ge 2$, endowed with the usual scalar product $(\cdot,\cdot)$ and norm
$\abs{\abs{\cdot}}$. We  write $B^n$ for the closed unit ball and denote by $\s^{n-1}$ the unit sphere in $\R^n$.
The unit sphere is endowed with the relative topology inherited from $\R^n$.
In particular, this applies to the interior or the boundary of a subset of the unit sphere.

\subsection{Convex bodies}\label{sII.1} Our general reference for the theory of convex
bodies is the book \cite{Schneider} by Schneider, to which we refer for all properties
of convex bodies mentioned in this section without proof.

We denote by $\K^n$ the
family of non-empty, compact, convex subsets (i.e.~convex bodies) of $\R^n$. If
$K$ and $L$ are convex bodies, the Minkowski sum (or vector sum) of $K$ and $L$ is
$$
K+L=\{a+b\,|\, a\in K\,,\, b\in L\},
$$
which is again a convex body. The same holds for the dilatation of a convex body
$K$ by a non-negative real $s\ge 0$, that is
$$
s\,K=\{s\,a\,|\, a\in K\}\,.
$$
The {\it support function} $h_K$ of a convex body $K$ is denoted by
$h_K\,:\,\sfe\to\R$ and given by
$$
h_K(u)=\sup_{x\in K}(x,u),\quad u\in\s^{n-1}\,.
$$
We will sometimes write $h$ instead of $h_K$, if $K$ is clear from the context.
If $h$ is the support function of a convex body, then the 1-homogeneous extension
of $h$ to $\R^n$ is convex. Conversely, if $H\,:\,\R^n\to\R$ is a 1-homogeneous convex function,
then its restriction to $\sfe$ is the support function of a convex body.
For all $K,L\in \K^n$ and $s,r\ge0$, we have
$$
h_{sK+rL}=sh_K+rh_L\,.
$$

\medskip

As usual, $\H^{j}$ denotes the $j$-dimensional Hausdorff measure in
$\R^n$ (normalized as in \cite{Federer1969}; in particular, $\mathcal{H}^n$ equals $n$-dimensional
Lebesgue measure). 
 For $K\in\K^n$, let $\partial K$
denote the topological boundary of $K$. For
$x\in\partial K$, we write ${\rm Nor}(K,x)$ for the normal cone of $K$ at
$x$. This non-empty closed convex cone consists of all outer normal vectors to supporting
half-spaces of $K$ passing through $x$.
Then we put ${\rm nor}(K,x):={\rm Nor}(K,x)\cap\s^{n-1}$.
For  $\omega\subseteq\sfe$, let
$$
\tau(K,\omega):=\{x\in\partial K\,:\,{\rm
nor}(K,x)\cap\omega\ne\emptyset\}
$$
be the set of all points
$x\in \partial K$ such that there exists an outer unit normal vector to
$K$ at $x$ contained in $\omega$. If $K$ has non-empty interior and
$\omega$ is a Borel subset of $\sfe$, then $\tau(K,\omega)$ is
$\H^{n-1}$-measurable (see \cite[\S 2.2]{Schneider}). In this case, the (surface) area
measure of $K$  can be defined by
$$
\Sa(K,\omega):=\H^{n-1}(\tau(K,\omega))
$$
for every Borel set $\omega\subseteq\sfe$.

\bigskip

Let $\langle\cdot,\cdot\rangle$ denote the Riemannian metric of $\sfe$ induced from $\R^n$, and let 
$\nabla$ denote the Levi-Civita connection. In the following, we consider local orthonormal frames of vector fields 
on $\s^{n-1}$, generically denoted by 
$\{E_1,\dots,E_{n-1}\}$. For a function $f\in\c^2(\sfe)$, we then write $f_i$ and
$f_{ij}$,
respectively, for the first and second  covariant derivatives
of $f$  with respect to $\{E_1,\dots,E_{n-1}\}$, where 
$i,j\in\{1,\dots,n-1\}$.  As usual,  $\delta_{ij}$ is the Kronecker symbol, hence $\delta_{ij}=\langle E_i,E_j\rangle$ 
for $i,j\in\{1,\dots,n-1\}$. Observe that  $f_i=\nabla_{E_i}f =E_i(f)$. To provide an invariant definition 
for some of the relevant notions to be considered subsequently, we recall that the gradient $\nabla f$ of $f$ is the uniquely 
determined vector field on $\sfe$ such that $\langle \nabla f,X\rangle=X(f)$,  for all vector fields $X$ on $\sfe$. 
 The Hessian form $\nabla^2 f$ is then defined as the field of bilinear 
forms on the tangent spaces $T_u\sfe$, $u\in\sfe$, of the unit sphere, which is determined by $\nabla^2f(X,Y)=\langle \nabla_X(\nabla f),
Y\rangle$, for all vector fields $X,Y$ on $\sfe$. For  $u\in \sfe$, the Hessian $\nabla^2f_u$ is a symmetric bilinear form 
on $T_u\sfe$ and $\nabla^2f(E_i,E_j)=f_{ij}$. The symmetry  of the matrix $\left(f_{ij}\right)_{i,j=1}^{n-1}$ will be crucial 
in the following. In particular, it ensures the existence of $n-1$ real eigenvalues, which are positive (non-negative) if and only if 
this matrix is positive definite (semi-definite). The symmetry is used implicitly, for instance, in the proof of Lemma  
\ref{lmIII.1}, and thus it is also essential for the subsequent lemmas. Note, however, that the third covariant 
derivatives are not completely symmetric (for $f\in C^3(\sfe)$). 
Using the Riemannian metric, we can identify $\nabla^2f$ with a field of 
symmetric linear maps of the tangent spaces of $\sfe$.

For $\phi\in \c^2(\s^{n-1})$,
$u\in\sfe$, and $i,j=1,\dots,n-1$, we put
$$
q_{ij}(\phi,u):=\phi_{ij}(u)+\delta_{ij}\phi(u)\,,
$$
where the covariant derivatives are computed
with respect to a local orthonormal frame (of vector fields),
and
$$
Q(\phi,u):=(q_{ij}(\phi,u))_{i,j=1}^{n-1}\,.
$$
All relevant quantities and conditions will be independent of
the particular choice of a local orthonormal frame in the following.
For the sake of brevity, we sometimes omit the variable $u$ and simply write
$q_{ij}(\phi)$ or $Q(\phi)$.

A convex body $K\in\K^n$ is said to be of class $\c^2_+$, if $\partial K$ is of class
$\c^2$ and the Gauss curvature is strictly positive at each point of
$\partial K$. If $K$ is of class $C^2_+$, then the Gauss map
$\nu_K\,:\,\partial K\to\sfe$, assigning to each point $x\in\partial K$
the outer unit normal to $\partial K$ at $x$, is a diffeomorphism between
$\partial K$ and $\sfe$. Moreover, the support function $h=h_K$ of $K$
belongs to $\c^2(\s^{n-1})$, and the $(n-1)\times(n-1)$ matrix
$Q(h,u)$ is positive definite for every $u\in\s^{n-1}$.  Conversely, if
$h\in\c^2(\sfe)$ is such that $Q(h,u)$ is positive definite (as usual, we then write $Q(h,u)>0$),
then $h$ is the support function of a (uniquely determined) convex body of
class $C^2_+$. In the following,  we consider the class of functions
$$
\Sf:=\left\{h\in\c^2(\s^{n-1}): Q(h,u)>0\;\mbox{
for every $u\in\s^{n-1}$}\right\},
$$
consisting of support functions of convex bodies of class $\c^2_+$ (cf.~\cite[\S 2.5]{Schneider}).

For $K\in\K^n$ of class $C^2_+$, the area measure of $K$ admits the
representation
\begin{equation}\label{I.0a}
\Sa(K,\omega)=\int_\omega\det(Q(h,u))\,\H^{n-1}(du)
\end{equation}
for every $\H^{n-1}$-measurable set $\omega\subseteq\sfe$. 

\begin{remark}\label{lateremark}
The representation \eqref{I.0a} 
is still valid for a convex body $K$ with support function $h\in C^2(\Omega)$ and
any measurable set  $\omega\subseteq\Omega$, where $\Omega\subseteq\sfe$ is open. This
follows by an application of the coarea formula to the differentiable map $\Omega\to\tau(K,\Omega)$,
$u\mapsto \text{grad } h(u)$, where $\text{grad } h$ is the Euclidean gradient of $h$. To see this, observe 
that the Jacobian of this map is  $Q(h)$ and  $\mathcal{H}^{n-1}$-almost all boundary points of $K$
have a unique exterior unit normal.
\end{remark}

\subsection{The cofactor matrix and a Lemma of Cheng and Yau}
\label{s:symm funct}
Let  $A=(a_{ij})_{i,j=1}^k$, $k\in\N$, be a real $k\times k$ matrix.
The determinant of  $A$ can be considered as a real-valued, polynomial
function of the entries $a_{ij}$. For $i,j=1,\dots,k$, we then  define
$$
c_{ij}[A]:=\frac{\partial\det}{\partial a_{ij}}(A)\,,
$$
and hence we can describe the cofactor matrix $C[A]$ of $A$ as
$$
C[A]=(c_{ij}[A])\,.
$$
In the following, we will mainly consider symmetric matrices. The set of
real, symmetric $k\times k$ matrices is denoted by $\Mat_k$. It is easy to see that if $A\in \Mat_k$,
then also $C[A]\in\Mat_k$.

\begin{remark}\label{cof 2x2 matrix}{\rm Consider  $A\in \Mat_2$ given by
\[
A=\left(
\begin{array}{cc}
a & b\\
b & c
\end{array}
\right)\,.
\]
Then the cofactor matrix of $A$ is
\[
C[A]=\left(\begin{array}{rr}
c & -b\\
-b & a
\end{array}
\right)\,.
\]
In particular, for real $2\times 2$ matrices $A,B$ we have the  useful linearity property
$$
C[A+B]=C[A]+C[B]\,.
$$
}
\end{remark}

In the next remark, we summarize some further properties of the cofactor matrix that will be
used later on.

\begin{remark}\label{propII.1}
\begin{itemize}
\item[{\rm (i)}] If $A$ is a real $k\times k$ matrix, then
$$
\det(A)=\frac{1}{k}\sum_{i,j=1}^k c_{ij}[A]a_{ij}\,.
$$
\item[{\rm (ii)}] Let $A\in\Mat_2$  be given. Then $C[A]$ is positive
\mbox{(semi-)}definite if and only if $A$ itself is positive \mbox{(semi-)}definite.
\end{itemize}
\end{remark}

A particularly useful feature of a matrix of type $C[Q(\phi)]$ is that 
for each row, that is, for fixed $i\in\{1,\ldots,n-1\}$, the sum of the covariant 
derivatives  $\left(c_{ij}[Q(h)]\right)_j$, for  $j=1,\ldots,n-1$, is zero. 
This fact was first observed and used by Cheng and Yau
\cite[(4.3)]{Cheng-Yau}. (See Lemma 1 in \cite{Colesanti-Saorin} for an extension. Relation 
(4.11) in \cite{Cheng-Yau} is also covered by Proposition 4,  page 
 5-8, and Lemma 18, page 7-45, in \cite{Spivak}.)

\begin{lema}\label{lmIII.1} Let $h\in\c^3(\sfe)$. Let $\{E_1,\ldots,E_{n-1}\}$ be
a local orthonormal frame of vector fields on $\sfe$. Then, for
$i\in \{1,\dots,n-1\}$, in the domain of the frame  we have
$$
\sum_{j=1}^{n-1} \left(
c_{ij}[Q(h)]\right)_j=0\,.
$$
\end{lema}

Let $h,\psi,\phi\in\c^3(\s^{n-1})$ be given. We then define a vector field $V$ on $\sfe$ by
$$
V=\sum_{i,j=1}^{n-1}\phi\,\psi_i\, c_{ij}[Q(h)]\, E_j\,,
$$
where $\{E_1,\ldots,E_{n-1}\}$ is a local orthonormal frame of vector fields.
Since the right-hand side  is independent of the choice of the orthonormal frame  (which
can be easily checked by a direct calculation), the vector field is globally defined.
Using Lemma \ref{lmIII.1}, we get for the divergence of this vector field that
\begin{align*}
\text{div }V&=\sum_{j=1}^{n-1}\left(\sum_{i=1}^{n-1}\phi\,\psi_i\, c_{ij}[Q(h)]\right)_j\\
&=\sum_{i,j=1}^{n-1}\phi_j\psi_i c_{ij}[Q(h)] +\phi\,\sum_{i,j=1}^{n-1}\psi_{ij} c_{ij}[Q(h)].
\end{align*}
Note that both summands on the right-hand side are independent of the choice of an 
orthonormal frame of vector fields (again
this can easily be checked).
The following lemma is  now an immediate consequence of the
divergence theorem on $\sfe$, applied to the vector field $V$, and a subsequent approximation argument.

\begin{lema}\label{l:interch int}
Let $h,\psi,\phi\in\c^2(\s^{n-1})$. Let $c_{ij}$
denote the entries of the matrix $C[Q(h)]$. Then, for
$i\in\{1,\dots,n-1\}$, we have
\begin{align*}
\int_{\s^{n-1}}\psi\sum_{i,j=1}^{n-1}\phi_{ij}\,c_{ij}\,d\H^{n-1}
&=-\int_{\s^{n-1}}\sum_{i,j=1}^{n-1}\phi_j\psi_i\,c_{ij}\,d\H^{n-1}\,\\
&=\int_{\s^{n-1}}\phi\sum_{i,j=1}^{n-1}\psi_{ij}\,c_{ij}\,d\H^{n-1}\,.
\end{align*}
\end{lema}

\begin{remark}\label{propII.neu}
{\rm The preceding lemma can also be derived by working with $1$-homogeneous extensions to $\R^n$ of
functions on $\sfe$, by establishing  a fact  analogous  to Lemma \ref{lmIII.1} in this setting, and
by applying the divergence theorem to a spherical shell. }
\end{remark}

\medskip

The following consequence of Remark \ref{propII.1} (i) and
Lemma \ref{l:interch int}
in the case $n=3$ will be needed subsequently.
Let $f,\phi\in\c^2(\s^{2})$. Then we have
\begin{align*}
&2\cdot \int_{\s^{2}}f(u)\det(Q(\phi,u))\, \mathcal{H}^2(du)\\
&=\int_{\s^{2}}f(u)\sum_{i,j=1}^2c_{ij}[Q(\phi,u)]q_{ij}(\phi,u)\, \mathcal{H}^2(du)\\
&=\int_{\s^{2}}f(u)\sum_{i,j=1}^2c_{ij}[Q(\phi,u)]\phi_{ij}(u)\, \mathcal{H}^2(du)\\
&\qquad+
\int_{\s^{2}}f(u)\phi(u)(\phi_{11}(u)+\phi_{22}(u)+2\phi(u))\, \mathcal{H}^2(du)\\
&=\int_{\s^{2}}\phi(u)\sum_{i,j=1}^2c_{ij}[Q(\phi,u)]f_{ij}(u)\, \mathcal{H}^2(du)\\
&\qquad +
\int_{\s^{2}}2f(u)\phi(u)^2 +f(u)\phi(u)(\phi_{11}(u)+\phi_{22}(u))\, \mathcal{H}^2(du)\\
&=\int_{\s^{2}}\phi(u)^2 [f_{11}(u)+f_{22}(u)+2f(u)]\, \mathcal{H}^2(du)\\
&\qquad + \int_{\s^{2}}\phi(u)\sum_{i,j=1}^2c_{ij}[Q(f,u)]\phi_{ij}(u)\, \mathcal{H}^2(du).
\end{align*}
By another application of Lemma \ref{l:interch int}, we obtain the next lemma which will
play a crucial role in the sequel.

\begin{lema}\label{inequ}
Let $f,\phi\in\c^2(\s^{2})$. Then
\begin{align*}
&2\cdot \int_{\s^{2}}f(u)\det(Q(\phi,u))\, \mathcal{H}^2(du)\\
&\qquad =\int_{\s^{2}}\phi(u)^2 \,\text{\rm trace}(Q(f,u))\, \mathcal{H}^2(du)
 - \int_{\s^{2}}\sum_{i,j=1}^2c_{ij}[Q(f,u)]\phi_{i}(u)\phi_{j}(u)\, \mathcal{H}^2(du).
\end{align*}
\end{lema}

\section{The $3$-dimensional case}\label{s: 3 dim case}
In this section, we prove the `only if' part of Theorem \ref{th dim 3}, which is the special case $n=3$ of Theorem
\ref{gral statement}. This also establishes
the initial step of the induction, which will be completed in Section \ref{sV}.

\begin{theorem}\label{th dim 3}
Let $f\in\c(\s^2)$ and let $\mathcal F\,:\,\K^3\to\R$ be defined by
$$
\F(K)=\int_{\s^2}f(u)\,\Ss(K,du)\,,\quad K\in\K^3\,.
$$
Then $\F$ satisfies
\begin{equation}\label{0.3bb}
\F((1-t)K_0+tK_1)\ge \min\{\F(K_0), \F(K_1)\}\,,
\end{equation}
for all $K_0,K_1\in\K^3$ and $t\in[0,1]$, if and only if 
$f$ is the support function of a convex body $L\in\K^3$.
\end{theorem}

\subsection{Preparatory steps}\label{dim3dss1}

The proof Theorem \ref{th dim 3} will require some preparations. Part of this preparatory material
is contained in the present subsection. In particular, we provide, for the reader's convenience,
an outline of a  proof for a simplified version of Theorem \ref{th dim 3}, under more restrictive assumptions on $f$.
We also point out the technical problems that arise in removing the additional assumptions on $f$
in order to cover the general
case. These problems are then settled in Section \ref{dim3dss3}, while in Section \ref{dim3dss4}
we complete the proof of Theorem \ref{th dim 3} in its full generality.

\medskip

To begin with, we assume that $\mathcal F$ satisfies a Brunn--Minkowski inequality of the form
(\ref{0.3}), namely
\begin{equation}\label{ipopos}
\mbox{${\mathcal F}(K)\ge0\,,\qquad$   $K\in\K^3$}\,,
\end{equation}
and
\begin{equation}\label{BM3d}
{\mathcal F}((1-t)K_0+tK_1)^{1/2}\geq (1-t) {\mathcal
F}(K_0)^{1/2}+t{\mathcal F}(K_1)^{1/2}\,,\quad
 \,K_0,K_1\in\K^3\,, \, t\in[0,1]\,.
\end{equation}

\medskip

Let $K\in\K^3$ and let $h$ be the support function of $K$. Let $\phi\in C(\s^2)$ and assume
that for some $\epsilon>0$ the function $h_s:=h+s\phi$ is a support function, for every
$s\in[-\epsilon,\epsilon]$. Let $K_s$ be the convex body having $h_s$ as its support function.
Hence the family of convex bodies $\{K_s\,:\,s\in[-\epsilon,\epsilon]\}$ provides a perturbation
of $K=K_0$. Let $F:[-\epsilon,\epsilon]\to\R_+$  be defined by $F(s):={\mathcal F}(K_s)$.

\begin{lema}\label{concavity F in lambda}
Under the above assumptions and notation, the function
$$
\sqrt{F}\,:\,[-\epsilon,\epsilon]\to\R_+\,,\qquad s\mapsto\sqrt{F(s)}\,,
$$
is concave.
\end{lema}
\begin{proof}  Let $s_1,s_2\in[-\epsilon,\epsilon]$ and $\lambda\in[0,1]$.
 Then we have $h_{(1-\lambda) s_1+\lambda s_2}=
(1-\lambda) h_{s_1}+\lambda h_{s_2}$, so that
$K_{(1-\lambda )s_1+\lambda s_2}=(1-\lambda) K_{s_1}+\lambda K_{s_2}$.
The conclusion now follows immediately from (\ref{BM3d}) and
the definition of $F$.
\end{proof}

\medskip

The following two remarks will be used in the sequel. Note that their validity is not restricted to
the three-dimensional case.

\begin{remark}\label{remark1}
{\rm Assume that $\mathcal F$ satisfies (\ref{ipopos}) and (\ref{0.3}) and
 does not vanish identically. Then
${\mathcal F}(L)>0$ for all $L\in\K^n$ of class $C^2_+$. Indeed, since convex bodies of class $C^2_+$
are dense in $\K^n$ and $\mathcal F$ is continuous, there exists  a convex body $K$ of class $C^2_+$ such that
${\mathcal F}(K)>0$.
As $K$ and $L$ are of class $C^2_+$, a suitable rescaled copy of $K$ is a {\it summand}
of $L$ (cf.~\cite[Cor.~3.2.10]{Schneider}), that is there exist $\lambda\in(0,1]$ and $M\in\K^n$ such
that $L=\lambda K+(1-\lambda)M$. By
(\ref{ipopos}) and (\ref{0.3}) we immediately get ${\mathcal F}(L)\ge{\mathcal F}(\lambda K)=\lambda^{n-1}\mathcal{F}(K)>0$.
In particular, we have ${\mathcal F}(B^n)>0$.}
\end{remark}

\medskip

\begin{remark}\label{remark2}{\rm Let $f,h\in C(\sfe)$. Then 
$$
\int_{\sfe}f(u)\, \Sa(K,du)=\int_{\sfe}h(u)\, \Sa(K,du)\,,\quad K\in\K^n\,,
$$ 
if and only if $f-h$ is the
restriction of a linear function to the unit sphere.
}
\end{remark}

\subsection{Outline of the proof of Theorem \ref{th dim 3} in a simplified case}\label{dim3dss2}
In this subsection we make several additional assumptions on the functional $\mathcal F$
(or rather, on $f$), and we outline the proof of Theorem \ref{gral statement}
in this special case.

We assume:
\begin{itemize}
\item[{\rm (i)}] {\it regularity:} $f\in C^2(\s^2)$;
\item[{\rm (ii)}] {\it symmetry:} $f$ is an even function, i.e. $f(u)=f(-u)$ for every $u\in\s^2$;
\item[{\rm (iii)}] {\it positivity:} $f>0$ on $\s^2$.
\end{itemize}
In particular, {\rm (iii)} implies that ${\mathcal F}(K)>0$ for every convex body $K$ with non-empty
interior. Clearly, we also assume that the corresponding functional $\mathcal F$ satisfies inequality
(\ref{BM3d}).

The support function of the unit ball $B^3$ is the
constant function $h\equiv1$ on $\s^2$. For a function $\psi\in C^2(\s^2)$ and $s\in\R$ consider the function
$h_s=1+s\psi$. Let $I$ denote the $2\times2$ identity matrix. Then, if $|s|$ is sufficiently small, the matrix
$Q(h_s,u)=I+sQ(\psi,u)$ is positive definite for every $u\in\s^2$. Hence
there exists $\epsilon>0$ such that $h_s\in{\mathfrak S}$ for every
$s\in[-\epsilon,\epsilon]$. Let $K_s$ be the convex body having $h_s$ as its support function
and define $F(s)={\mathcal F}(K_s)$. According to Lemma \ref{concavity F in lambda},
$\sqrt{F}$ is concave. Since $F(s)>0$ and $F$ is twice differentiable, we obtain that
\begin{equation}\label{dimsempl1}
2F(0)F''(0)-(F'(0))^2\le0\,.
\end{equation}
From (\ref{I.0a}) we now conclude that
$$
F(s)=\int_{\s^2}f(u)\,\det(Q(h_s,u))\,\H^2(du)=
\int_{\s^2}f(u)\,\det(I+sQ(\psi,u))\,\H^2(du)\,.
$$
Differentiating with respect to $s$,  at $s=0$ we get that
\begin{equation*}
F'(0)=\int_{\s^2}f(u)\,{\rm trace}(Q(\psi,u))\,\H^2(du)\,,\quad
F''(0)=2\,\int_{\s^2}f(u)\,\det(Q(\psi,u))\,\H^2(du)\,.
\end{equation*}
Assume that $\psi$ is odd, that is $\psi(u)=-\psi(-u)$ for  $u\in\s^2$.
Then ${\rm trace}(Q(\psi,\cdot))$ is odd as well and, by {\rm (ii)},
it follows that $F'(0)=0$. Hence,
by (\ref{dimsempl1}) and since $\mathcal F (0)>0$,
we get $F''(0)\le0$, i.e.
\begin{equation}\label{dimsempl2}
\int_{\s^2}f(u)\,\det(Q(\psi,u))\,\H^2(du)\le 0
\end{equation}
for every odd function $\psi\in C^2(\s^2)$.
We now want to remove the assumption of being odd on the test function,
at the price of reducing its support. Let $\phi\in C^2(\s^2)$ be such that
its support is contained in an open hemisphere $\mathcal E$ of $\s^2$, and define
$$
\psi(u)=
\begin{cases}
\phi(u),& \text{ if }u\in{\mathcal E},\\
-\phi(-u),& \text{ if }u\notin{\mathcal E}.
\end{cases}
$$
Clearly, $\psi$ is well defined on $\s^2$ and zero close to the  boundary of $\mathcal{E}$. Since
 $\psi$ is odd and $\det(Q(\psi,\cdot))$ is even, we deduce from (\ref{dimsempl2}) that
\begin{equation}\label{dimsempl3}
\int_{\s^2}f\,\det(Q(\phi))\,d\H^2\le 0
\end{equation}
for any $\phi\in C^2(\s^2)$ whose support is contained in an open hemisphere.
Writing $(c_{ij})$ for the matrix $C[Q(f)]$, we now obtain from Lemma \ref{inequ} that
\begin{eqnarray}\label{dimsempl4}
\int_{\s^2}\phi^2\,{\rm trace}(c_{ij})\,d\H^2
\le\int_{\s^2}\sum_{i,j=1}^2c_{ij}\phi_i\phi_j\,d\H^2\,,
\end{eqnarray}
for any such $\phi\in C^2(\s^2)$.
This is a functional inequality of Poincar\'e type on $\s^2$. It is rather intuitive that
such an inequality can be valid for any $\phi$ as described above only if the matrix $(c_{ij})$
is positive semi-definite
throughout $\s^2$. This fact is demonstrated in Lemma \ref{propIV.2}. The idea of the proof is
that if at some point $u_0\in \s^2$ the matrix $(c_{ij})$ admits an eigenvector $e$ with
negative eigenvalue, we may construct
a sequence of admissible test functions $\phi_k$, $k\in \N$, such that the
$L^\infty(\s^2)$--norms of these functions converge to zero and  their gradients
tend to be parallel to $e$, of constant unit norm in a neighbourhood
of $u_0$ and zero everywhere else. Choosing $\phi=\phi_k$ in (\ref{dimsempl4}), and letting $k$ tend to
infinity, we get a contradiction.

Once we know that $(c_{ij})$ is positive semi-definite, the same is true for $Q(f)$,
and then $f$ is a support function
(see Corollary \ref{appendixcor1}).

In the next subsection, and in the Appendix, we prove some results which permit us to adapt the
above idea when the assumptions (i), (ii), (iii) are not imposed. In particular, to remove symmetry,
 in the preceding perturbation argument we replace the ball by a parametric ``spherical cone'' $C$.
 This allows us to cancel the
term $F'(0)$ in (\ref{dimsempl1}). Then, in order to prove (\ref{dimsempl3}), we need to know that
$F(0)={\mathcal F}(C)>0$ for a suitable cone. A corresponding fact is provided in Lemma \ref{techlemma3}. Finally, the
regularity assumption will be removed by a standard approximation procedure (described in the Appendix),
which will enter at the level of (\ref{dimsempl3}).

\subsection{Some technical lemmas}\label{dim3dss3}

Let $\theta\in[0,\pi/2]$ and $P\in\s^2$. We will denote by
$I_{\theta}(P)$ the set of points in $\s^2$ with spherical distance to $P$
less than or equal to $\theta$. More explicitly,
$$
I_\theta(P):=\{Q\in\s^2\,:\,(P,Q)\ge\cos\theta\}
$$
is a spherical cap of $\s^2$ with angle of aperture $\theta$ around $P$.
Let $C(P,\theta)$ be the convex hull of $\{0\}$ and $I_{\theta}(P)$.
Equivalently, if $D$ is the cone $D:=\{tx\,:\,x\in I_\theta(P)\,,\,t\ge0\}$,
then $C(P,\theta)$ is the intersection of $D$ with
the unit ball centered at the origin. Clearly, $C(P,\theta)$ degenerates
for $\theta=0$  into a segment with end-points $P$ and $\{0\}$, and
it coincides with a half ball for $\theta=\pi/2$. The area measure of $C(P,\theta)$
is described in the next lemma. Recall that ${\rm nor}(K,x)$ is the set of
exterior unit normals of the convex body $K$ at $x\in \partial K$.

\begin{lema}\label{lema S2}
For  $P\in\s^2$ and  $\theta\in [0,\pi/2)$, we have
\[
\Ss\left(C(P,\theta),\cdot\right)=\H^2(\cdot)\meares{I_{\theta}(P)}+
\frac{\tan{\theta}}{2}\,\H^1(\cdot)\meares{\Gamma_{\theta}(P)}\,,
\]
where
\[
\Gamma_{\theta}(P):=\{x\in\s^2: (x,P)=-\sin{\theta}\}
\]
and $\H^2(\cdot)\meares{I_{\theta}(P)}$ and $\H^1(\cdot)\meares{\Gamma_{\theta}(P)}$ denote the restrictions of the
measures $\H^2$ and $\H^1$ to $I_{\theta}(P)$ and $\Gamma_{\theta}(P)$, respectively.
\end{lema}
\begin{proof} 
The case $\theta=0$ is clear, hence assume that $\theta\in(0,\pi/2)$.
As the area measure commutes with rotations (see \cite{Schneider}, p. 205), we may assume that
$P=(0,0,1)$. Hence
$$
C(P,\theta)=\{(r\sin\theta'\cos\varphi, r\sin\theta'\sin\varphi,r\cos\theta' )\,:\,0\le r\le 1,\,0\le\theta'\le\theta,\,
0\le\varphi\le2\pi\}\,.
$$
Next we identify the two relevant portions of the boundary of $C(P,\theta)$:
the spherical cap $A_1$ and the
conical surface $A_2$,
\begin{eqnarray*}
A_1&=&\{(\sin\theta'\cos\varphi,\sin\theta'\sin\varphi,\cos\theta')\,:\,0\le\theta'<\theta,\,
0\le\varphi\le2\pi\}\,,\\
A_2&=&\{r(\sin\theta\cos\varphi,\sin\theta\sin\varphi,\cos\theta)\,:\,0<r<1,\,
0\le\varphi\le2\pi\}\,.
\end{eqnarray*}
In particular, $I_{\theta}(P)$ is the closure of $A_1$. Hence $A_1\cup A_2\subseteq\partial C(P,\theta)$ and
\begin{equation}\label{appendix7}
\H^2(\partial C(P,\theta)\setminus(A_1\cup A_2))=0\,.
\end{equation}
Note that $\partial C(P,\theta)$ is differentiable at each point of $A_1\cup A_2$. For $Q\in A_1\cup A_2$, let $\nu(Q)$ be
the outer unit normal to $\partial C(P,\theta)$ at $Q$, i.e.~${\rm nor}(C(P,\theta),Q)=\{\nu(Q)\}$.
If $Q\in A_1$, then $\nu(Q)=Q$. If $Q=r(\sin\theta\cos\varphi,\sin\theta\sin\varphi,\cos\theta)\in A_2$, then
$$
\nu(Q)=(\cos\theta\cos\varphi,\cos\theta\sin\varphi,-\sin\theta)\,.
$$
Hence $\nu(A_1\cup A_2)=A_1\cup\Gamma_\theta (P)$. By (\ref{appendix7}) and the definition of the area measure
this proves that for every Borel subset $\omega$ of $\s^2$ we have
$\Ss(C(P,\theta),\omega)=\Ss(C(P,\theta),\omega_1\cup\omega_2)$,
where $\omega_1=\omega\cap A_1$ and $\omega_2=\omega\cap \Gamma_\theta(P)$, and thus
\begin{eqnarray*}
\Ss(C(P,\theta),\omega)&=&\Ss(C(P,\theta),\omega_1\cup \omega_2)
=\Ss(C(P,\theta),\omega_1)+\Ss(C(P,\theta),\omega_2)\\
&=&\H^2(\nu^{-1}(\omega_1))+\H^2(\nu^{-1}(\omega_2))\\
&=&\H^2(\omega_1)+\H^2(\nu^{-1}(\omega_2))\,.
\end{eqnarray*}
Finally, it is not hard to check that
$$
\H^2(\nu^{-1}(\omega_2))=\frac{\tan{\theta}}{2}\,\H^1(\omega_2)\,.
$$
\end{proof}


\begin{lema}\label{propIV.2}
Let $f\in C^{2}(\s^{2})$, $P\in\s^2$ and $\theta\in(0,\pi/2)$. If,
for every $\phi\in\c^\infty(\s^2)$ with support  contained in $I_{\theta}(P)$, we have
\begin{equation}\label{pippo2}
\int_{\s^2}{f(u)\det\left(Q(\phi,u)\right)\,\H^2(du)}\leq 0\,,
\end{equation}
then $Q(f,u)$ is positive semi-definite for every $u\in I_\theta(P)$.
\end{lema}

\begin{proof}
By Remark \ref{propII.1} (ii) and a continuity argument, it is sufficient to show that
$C[Q(f,u)]$ is positive semi-definite for every $u$ in the  interior of $I_\theta(P)$.
From the assumption (\ref{pippo2}) we deduce by means of Lemma \ref{inequ} that
\begin{equation}\label{int f det leq 0 rewritten}
\int_{\s^2}{\phi^2{\rm trace}(C[Q(f)])\,d\H^2}\leq
\int_{\s^2}{\sum_{i,j=1}^2{c_{ij}[Q(f)]\phi_i\phi_j}\,d\H^2} \,.
\end{equation}
By a standard approximation argument \eqref{int f det leq 0 rewritten} can be extended to every function
$\phi$, with support contained in $I_\theta(P)$, which is merely Lipschitz on $\s^2$
(correspondingly, the first derivatives of $\phi$ will be defined $\H^{2}$-a.e.
on $\s^2$).

For the sake of brevity, we define $c_{ij}(u):=c_{ij}[Q(f,u)]$ for $u\in\s^2$. Arguing by
contradiction, let us assume that there exists some  $\bar u$ in the
interior of $ I_\theta(P)$ and a vector
$v=(v_1,v_2)\ne0$ such that
\[
\sum_{i,j=1}^{2}c_{ij}[Q(f,\bar u)]v_iv_j < 0\,.
\]
Without loss of generality (by a proper choice of the coordinate system),
we may assume that $\bar u=(0,0,1)\in\s^2$  and $v=(1,0)$. Then
\begin{equation}\label{IV.0aaa}
\sum_{i,j=1}^{2}c_{ij}({\bar u})v_iv_j=c_{11}({\bar u})<0\,.
\end{equation}
We identify $H:=\{x=(x_1,x_2,x_3)\in\R^3\,:\,x_3=0\}$ with $\R^{2}$ and,
for $r\in(0,1)$, we set
\begin{align*}
D_r:&=\{(x_1,x_2)\in\R^2:\abs{x_i}\le r\,,\,i=1,2\}\,,\\
\tilde D_r:&=\{u=(u_1,u_2,u_3)\in\s^2\,:\,u_3>0\,,\,(u_1,u_2)\in
D_r\}\,.
\end{align*}
Using (\ref{IV.0aaa}), we will construct a Lipschitz function $\phi$ with support contained 
in $I_\theta(P)$, but such that
inequality \eqref{int f det leq 0 rewritten} fails to be true. In order to obtain
such a function, we first define $\bar g\,:\,[-1,1]\to \R_+ $ by $\bar
g(t)=1-\abs{t}$, and denote by $g\,:\,\R\to \R_+$ the periodic
extension of $\bar g$ to the whole real line. Let $\epsilon>0$ and define
$g_{\epsilon}(x)=\epsilon g(x/\epsilon)$. Notice that
$g_{\epsilon}\to 0$ uniformly on $\R$, as $\epsilon\to 0^+$.
In the following, by writing $\epsilon\to 0^+$ we mean that $\epsilon$
runs through a decreasing sequence which converges to zero.
Let $G\,:\,\R\to \R$ be defined by
\[
G(t):=\left\{
\begin{array}{ll} 1\,, & \text{for } t\in[-1/2,1/2]\,, \\
                          0\,, & \text{for } |t|\geq 1\,,\\
                         1-2|t|\,, & \text{otherwise\,.}
\end{array}\right.
\]
Hence $G$ is bounded and Lipschitz. Let us fix $r\in(0,1)$ for the moment (the choice of $r$ 
will be adjusted subsequently, but $r$ will be bounded away from $0$, independent of $\epsilon$). The function
\begin{eqnarray*}
\Phi_{\epsilon}(x_1,x_2)=
g_{\epsilon}(x_1)G(x_1/r)G(x_2/r)\,,\quad(x_1,x_2)\in D_r\,,
\end{eqnarray*}
satisfies ${\rm sprt}(\Phi_{\epsilon})\subseteq D_r$ and is Lipschitz
on $D_r$. We have
\[
\frac{\partial\Phi_{\epsilon}}{\partial x_2}(x_1,x_2)=\frac{1}{r}
g_{\epsilon}(x_1) G'(x_2/r)G(x_1/r)\,\quad \mbox{for $\H^2$-a.e.~$(x_1,x_2)\in D_r$\,.}
\]
As $0\le G\le 1$, $\abs{G'}\le 2$ and $\abs{g_\epsilon}\le\epsilon$ in $\R$,
$$
\abs{\frac{\partial\Phi_\epsilon}{\partial x_2}}
\le\frac{2\epsilon}{r}\,,\quad\mbox{$\H^2$-a.e.~in $D_r$\,,}
$$
and then
\begin{equation}
\label{IV.0aa} \lim_{\epsilon\to0^+}\frac{\partial\Phi_\epsilon}{\partial
x_2}=0\,,\quad \mbox{$\H^2$-a.e.~in $D_r$.}
\end{equation}
On the other hand, for $\H^2$-a.e. $(x_1,x_2)\in D_r$ we have
\[
\frac{\partial\Phi_\epsilon}{\partial x_1}(x_1,x_2)=
\frac{1}{r}g_\epsilon(x_1)G'(x_1/r)G(x_2/r)+g_{\epsilon}'(x_1)G(x_1/r)G(x_2/r)\,.
\]
As $\abs{g'_\epsilon}=1$ holds $\H^1$-a.e. in $\R$, it follows that
\begin{equation}\label{IV.0b}
\lim_{\epsilon\to0^+}\left|{\frac{\partial\Phi_\epsilon}{\partial
x_1}}(x_1,x_2)\right|=G(x_1/r)G(x_2/r)\quad\mbox{for $\H^{2}$-a.e.
$(x_1,x_2)\in D_r$}\,.
\end{equation}
In particular, the above limit equals $1$, $\H^{2}$-a.e.~in $D_{r/2}$.
Moreover, we have
$$
\left|{\frac{\partial\Phi_\epsilon}{\partial x_1}}(x_1,x_2)\right|\le \frac{2\epsilon}{r}+1
\quad \mbox{for $\H^2$-a.e. $(x_1,x_2)\in D_r$}\,.
$$
Next, consider the function
$$
\phi_\epsilon(u)=\phi_{\epsilon}(u_1,u_2,u_3):=\Phi_{\epsilon}(u_1,u_2)\,,\quad
\,u\in\tilde D_r\,,
$$
and extend $\phi_\epsilon$ to be zero in the rest of the unit sphere
$\s^2$. As $\bar u$ is in the interior of $ I_\theta(P)$,
if $r$ is sufficiently small, then the support of
$\phi_\epsilon$ is contained in $I_{\theta}(P)$.
In the sequel, for $u=(u_1,u_2,u_3)\in\tilde D_r$ we set
$u'=(u_1,u_2)\in D_r$, i.e., $u=(u',u_3)$.
We may choose $r$ small enough that there exists a local orthonormal
frame  on $\tilde D_r$. Taking covariant derivatives with
respect to this frame, by \eqref{int f det leq 0 rewritten} we have
\begin{equation}\label{previnequality}
\int_{\s^2}{\phi_{\epsilon}^2\,{\rm trace}\left(C[Q(f)]\right)\,d\H^2}\leq
\int_{\s^{2}}{\sum_{i,j=1}^{2}\,c_{ij}\left[Q(f)\right](\phi_{\epsilon})_i(\phi_{\epsilon})_j\,d\H^{2}}\,.
\end{equation}
Since $\Phi_\epsilon$ converges to zero uniformly as $\epsilon\to 0^+$,
the same is valid for $\phi_\epsilon$. Hence, taking limits on both sides of \eqref{previnequality}, we get
\begin{equation}
\label{IV.0c} 0\leq
\liminf_{\epsilon\to0^+}\int_{\s^2}{\sum_{i,j=1}^{2}{c_{ij}\,(\phi_{\epsilon})_i(\phi_{\epsilon})_j}}\,d\H^2\,.
\end{equation}
The covariant derivatives of $\phi_{\epsilon}$ can be computed in terms of
partial derivatives of $\Phi_\epsilon$ with respect to Cartesian
coordinates on $D_r$. In particular, for $u\in\tilde D_r$  there exists a $2\times 2$ matrix
$(\gamma_{ij}(u))_{i,j=1}^2$ with
$\gamma_{ij}\in\c^\infty( \tilde D_r)$, for $i,j=1,2$, such that
\[
(\phi_{\epsilon})_i(u)=
\sum_{k=1}^{2}\gamma_{ik}(u)\frac{\partial\Phi_\epsilon}{\partial
u_k}(u')\,\quad\mbox{for $\H^{2}$-a.e. $u\in \tilde D_r$\,.}
\]
We may assume that the local orthonormal frame has been chosen so that
$(\gamma_{ij}(\bar{u}))_{i,j=1}^2$ is the identity matrix.
Then, for $\H^2$-a.e. $u\in\tilde D_r$
\[
\sum_{i,j=1}^{2}{c_{ij}(u)(\phi_{\epsilon})_i(u)(\phi_{\epsilon})_j(u)}=
\sum_{i,j=1}^{2}\sum_{k,l=1}^{2}c_{ij}(u)\gamma_{ik}(u)\gamma_{jl}(u)\frac{\partial\Phi_\epsilon}{\partial
u_k}(u')\frac{\partial\Phi_\epsilon}{\partial u_l}(u')\,.
\]
This expression is bounded, in absolute value, by the boundedness of the
partial derivatives of $\Phi_\epsilon$. Moreover, by (\ref{IV.0aa}) and
(\ref{IV.0b}), for $\H^2$-a.e. $u\in\tilde D_r$ we have
$$
\lim_{\epsilon\to0^+}
\sum_{i,j=1}^{2}{c_{ij}(u)(\phi_{\epsilon})_i(u)(\phi_{\epsilon})_j(u)}=
G^2(u_1/r)G^2(u_2/r)
\sum_{i,j=1}^{2}c_{ij}(u)\gamma_{i1}(u)\gamma_{j1}(u)\,.
$$
Note that
$$
\sum_{i,j=1}^{2}c_{ij}(\bar u)\gamma_{i1}(\bar u)\gamma_{j1}(\bar
u)=c_{11}(\bar u)<0\,.
$$
Consequently, we may choose $r$ sufficiently small so that
$$
\sum_{i,j=1}^{2}c_{ij}(u)\gamma_{i1}(u)\gamma_{j1}(u)\le
c<0\,,\quad\, u\in\tilde D_r\,.
$$
Then, by the dominated convergence theorem, we have
\begin{align*}
&\lim_{\epsilon\to0^+}\int_{\s^2}\sum_{i,j=1}^{2}{c_{ij}(u)(\phi_{\epsilon})_i(u)(\phi_{\epsilon})_j}(u)\,
\H^2(du)\\
&\qquad =\int_{\s^2}G^2(u_1/r)G^2(u_2/r)\sum_{i,j=1}^{2}c_{ij}(u)\gamma_{i1}(u)\gamma_{j1}(u)\,\H^{2}(du)\\
&\qquad =\int_{\tilde D_r}G^2(u_1/r)G^2(u_2/r)\sum_{i,j=1}^{2}c_{ij}(u)\gamma_{i1}(u)\gamma_{j1}(u)\,\H^{2}(du)\\
&\qquad \le c\cdot \int_{\tilde D_r}G^2(u_1/r)G^2(u_2/r) \,\H^{2}(du)<0\,.
\end{align*}
This is in contradiction with (\ref{IV.0c}).
\end{proof}

\begin{lema}\label{techlemma3}
Let $f\in C(\sfe)$. 
Assume that $\mathcal{F}$ is not identically zero and satisfies (\ref{ipopos}) and
(\ref{BM3d}). Then, for every $P\in\s^2$ there exists some $\theta\in(0,\pi/2)$ so that
\[
\mathcal{F}\left(C(P,\theta)\right)>0.
\]
\end{lema}

\begin{proof}
We argue by contradiction. Assume that there exists some $P\in\s^2$
such that for all $\theta\in(0,\pi/2)$, we have $\mathcal{F}\left(C(P,\theta)\right)=0$.

Let us fix $\theta$, for the moment, and denote by $K$ the set $C(P,\theta)$ and by $h$ its support function. Then
$h\equiv1$ in $I_{\theta}(P)$.
Let $\phi\in C^{\infty}(\s^2)$ with support
contained in the  interior $I_\theta(P)^\circ$ of $I_{\theta}(P)$. By Proposition \ref{appendixprop1} there exists some
$\epsilon> 0$ such that
$h_s=h+s\phi$ is a support function for every $s\in [-\epsilon,\epsilon]$. Let $K_s$
denote the convex body whose support function is $h_s$
and $F(s)={\mathcal F}(K_s)$. Then $F(0)={\mathcal F}(K)=0$ and (\ref{ipopos}) yields that $F$
has a minimum at $s=0$.
Thus, if the derivatives of $F$ exist, then
\[
\mbox{$F'(0)=0\quad$ and $\quad F''(0)\geq 0$}.
\]
In order to obtain a suitable expression for the second derivative of $F$ at $s=0$, we first observe that
$$
F(s)=\int_{I_\theta(P)^\circ} f(u) \, \Ss(K_s,du)+\int_{\mathbb{S}^2\setminus I_\theta(P)^\circ} f(u) \, \Ss(K_s,du)\,.
$$
By the definition of $K_s$, in 
particular, since the support of $\phi$ is contained in $I_\theta(P)^\circ$, \cite[Theorem 1.7.4]{Schneider}  yields that $\tau(K_s,\omega)=\tau(C(P,\theta),\omega)$ for all Borel sets $\omega\subset \mathbb{S}^2\setminus I_\theta(P)^\circ$. 
Since area measures are locally defined (see \cite[p.~206]{Schneider}), we conclude that 
$\Ss(K_s,\cdot)=\Ss(C(P,\theta),\cdot)$ on  $\mathbb{S}^2\setminus I_\theta(P)^\circ$. Hence, Lemma \ref{lema S2} yields 
that 
$$
\int_{\mathbb{S}^2\setminus I_\theta(P)^\circ} f(u) \, \Ss(K_s,du)=\frac{\tan\theta}{2}\int_{\Gamma_\theta(P)}{f
\,d\H^1}.
$$
On $I_\theta(P)$, the support function $h(K_s,\cdot)$ of $K_s$ is of class $C^2$. 
Hence, by Remark \ref{lateremark} we 
get 
$$
\Ss(K_s,\omega)=\int_\omega \det \left(Q(h_s)\right)d\H^2\,,
$$
for all Borel sets $\omega\subset I_\theta(P)^\circ$. Since $\H^2(I_{\theta}(P)\setminus I_{\theta}(P)^\circ)=0$,  we finally arrive at
\begin{equation}\label{dim 3 F(s)}
F(s)=\int_{I_{\theta}(P)}{f
\det\left(Q(h_s)\right)d\H^2}+\frac{\tan\theta}{2}\int_{\Gamma_\theta(P)}{f
\,d\H^1}.
\end{equation}
From \eqref{dim 3 F(s)} we see that $F$ is indeed twice differentiable. Moreover,
 the second term is independent of $s$.
Using the definition of the cofactor matrix, it thus follows that
\begin{equation}\label{dim3 F's}
F'(s)=\int_{I_{\theta}(P)}{f
\sum_{i,j=1}^2 c_{ij}\left[Q(h_s)\right]q_{ij}(\phi)\,d\H^2}.
\end{equation}
Since $h=1$ in $I_{\theta}(P)$, the cofactor matrix of $Q(h,u)$ is the identity matrix,
for every $u\in I_\theta(P)$. This implies that
\begin{equation}\label{dim3 F'=0}
F'(0)=\int_{I_{\theta}(P)}f\,{\rm trace}(Q(\phi))\,d\H^2=0\,.
\end{equation}
Differentiating (\ref{dim3 F's}), we obtain for the second derivative of $F$ at $s=0$ the expression
\begin{equation}\label{dim3 F''}
F''(0)=2\int_{I_{\theta}(P)}{f\det\left(Q(\phi)\right)\,d\H^2}.
\end{equation}

Now we use the regularization argument described in the Appendix. Let
$(f_k)_{k\in\N}$ be the sequence of functions in  $C^{\infty}(\s^2)$,
converging uniformly to $f$ on $\s^2$, that is constructed in Lemma \ref{l:regularization}.
Let $C\subset\s^2$ be a compact set contained in the  interior of $I_\theta(P)$,
and let $\psi\in C^\infty(\s^2)$ be such that its support
is contained in $C$. Let $g(u)={\rm trace}(Q(\psi(u))$. It is clear that ${\rm
sprt}(g) \subseteq C$. Then
\begin{equation}\label{pippo}
\begin{split}
\int_{\s^2}{f_k(u)g(u)\,\H^2(du)}&=\int_{\s^2}{\left(\int_{{\bf O}(2)}{f(\rho
u)\omega_k(\rho)\,\nu(d\rho)}\right)g(u)\,\H^2(du)}\\
&=\int_{{\bf O}(2)}{\omega_k(\rho)}\int_{\s^2}{f(\rho u)g(u)\, \H^2(du)\,
\nu(d\rho)}\\
&=\int_{\{||\rho-{\rm id}||<\delta_k\}}{\omega_k(\rho)\left(\int_{\s^2}{f(\rho
u)g(u)\,\H^2(du)}\right)\,\nu(d\rho)}\,,
\end{split}
\end{equation}
where ${\rm id}$ is the identity element of ${\bf O}(2)$ and $\delta_k=\frac{1}{k^2}$
(see the definition of $\omega_k$ in the Appendix).
Fix $k\in\N$ and $\rho\in{\bf O}(2)$ such that $\|\rho-{\rm id}\|<\delta_k$, and let
$\psi_\rho$ be defined by $\psi_\rho(u)=\psi(\rho^{-1}(u))$ for $u\in \s^2$.
Then, by the rotation invariance of the Hausdorff measure $\H^2$ 
and Lemma \ref{appendixlemma2} we get
\[
\int_{\s^2}{f(\rho u)g(u)
\,\H^2(du)}=\int_{\s^2}{f(u)\,
{\rm trace}\left(Q(\psi_{\rho},u)\right)
\,\H^2(du)}\,.
\]
For sufficiently large $k$ (this depends on the choice of the set $C$, of course),
the support of $\psi_{\rho}$
is contained in $I_{\theta}(P)$, hence we may apply \eqref{dim3 F'=0}
with $\phi=\psi_{\rho}$ and get
\[
\int_{\s^2}{f(u)\,
{\rm trace}\left(Q(\psi_{\rho},u)\right)
\,\H^2(du)}=0\,.
\]
Thus, using (\ref{pippo}), we arrive at
$$
\int_{\s^2}{f_k(u)
{\rm trace}\left(Q(\psi,u)\right)
\,\H^2(du)}=0
$$
for every $\psi\in C^{\infty}(\s^2)$ with ${\rm
sprt}(\psi)\subseteq C$, and for sufficiently large $k$.
As $f_k$ is smooth we get, using integration by parts (that is, a special case of Lemma \ref{l:interch int}),
\begin{equation}\label{1. deriv 0+lema}
\int_{\s^2}{f_k(u)\,
{\rm trace}\left(Q(\psi,u)\right)
\,\H^2(du)}
=
\int_{\s^2}\psi(u)\,{\rm trace}(Q(f_k,u))\,\H^2(du)
=0\,.
\end{equation}
From \eqref{1. deriv 0+lema} (and the regularity of $f_k$) it follows  that
\begin{equation}\label{pippo4}
{\rm trace}(Q(f_k,u))=0\,,
\end{equation}
for all $u\in C$ and all  sufficiently large $k\in\N$.

\medskip

Performing the same argument now with
$g=\det(Q(\psi))$, and using Lemma \ref{appendixlemma2},
(\ref{dim3 F''}) and $F''(0)\ge0$, we obtain that
$$
\int_{\s^2}{f_k(u)\det
\left(Q(\psi,u)\right)\,\H^2(du)}\geq 0\,,
$$
for every $\psi \in C^{\infty}(\s^2)$ such that ${\rm
sprt}(\psi)\subseteq C$, and for all sufficiently large $k$.
In particular, for any fixed $\theta'\in (0,\theta)$, we may choose
$C=I_{\theta'}(P)$. For this choice of $C$ we now apply Lemma \ref{propIV.2}
to $-f_k$ and conclude that $Q(-f_k,u)$
is positive semi-definite for all $ u\in I_{\theta'}(P)$,
if $k$ is sufficiently large. But then (\ref{pippo4}) with $f_k$ replaced by $-f_k$ implies
that $Q(-f_k)=0=Q(f_k)$ in $I_{\theta'}(P)$. Now Corollary \ref{appendixcor2} shows that $f_k$ is
the restriction of a linear function to $I_{\theta'}(P)$ if  $k$ is large enough.
Letting $k$ tend to infinity, we obtain that the same conclusion holds for $f$ and,
as $\theta$ was arbitrary, we finally conclude that $f$ is linear on the hemisphere
$I_{\pi/2}(P)$. According to Remark \ref{remark2},
we may assume that $f=0$ in $I_{\pi/2}(P)$. Then,
\eqref{dim 3 F(s)} and $F(0)=0$ yield
$$
\int_{\Gamma_\theta(P)}{f \,d\H^1}=0\,,\quad\,\theta \in (0,\pi/2)\,.
$$
A suitable decomposition of spherical Lebesgue measure now implies that
$$
\int_{\s^2\backslash I_{\pi/2}(P)}{f \,d\H^2}=0.
$$
On the other hand, for the unit ball we have
\[
\F(B^3)=\int_{\s^2}{f\, d\H^2}=\int_{I_{\pi/2}(P)}{f\,
d\H^2}+\int_{\s^2\backslash I_{\pi/2}(P)}{f \,d\H^2}=0\,,
\]
in contradiction to Remark \ref{remark1}.
\end{proof}

\begin{prop}\label{prop int f det<0}
Let $f\in C(\mathbb{S}^{n-1})$. 
Assume that $\F$ is not identically zero and satisfies (\ref{ipopos}) and (\ref{BM3d}). Then, for each point
$P\in \s^2$ there is some $\theta\in(0,\pi/2)$ such that, for all
$\phi\in C^{2}(\s^2)$ with support contained in the  interior of $I_{\theta}(P)$, we have
$$
\int_{\s^2}{f(u) \det\left(Q(\phi,u)\right)\,\H^2(du)}\leq
0.
$$
\end{prop}

\begin{proof}
Let $P\in\s^2$ be given.
Let $\overline{P}$ be the antipodal point of $P$, and let $\overline{\theta}\in
(0,\pi/2)$ be such that $\F(C(\overline{P},\overline{\theta}))>0$. By Lemma \ref{techlemma3}
the existence of $\overline{\theta}$ is ensured. Now we define  $\theta:=\frac{\pi}{2}-\overline{\theta}$ and
$\Omega:=I_\theta(P)^\circ$.
Let $h$
denote the support function of $C(\overline{P},\overline{\theta})$.
Clearly, $h\equiv 0$ in $\Omega$. We now consider $\phi\in C^{2}(\s^2)$ with support contained in $\Omega$.

Let $\eta>0$ and define the convex body $K_\eta$   by
$$
K_\eta:=C(\overline{P},\overline{\theta})+\eta B^3\,.
$$
Let $h_\eta$ be the support function of
$K_\eta$. Then, for $u\in\Omega$,
$$
h_\eta(u)=h(u)+\eta h_{B^3}(u)=\eta\,.
$$
Moreover, the assumptions of Proposition \ref{appendixprop1} are fulfilled
by $h_\eta$ and $\phi$, hence there exists some $\epsilon>0$ (which may depend on $\eta$ as well)
such that for every $s\in[-\epsilon,\epsilon]$ the function $\hat h_{\eta,s}:=h_\eta+s\phi$ is
the support function of a convex body $\hat K_{\eta,s}$. Let $F_\eta\,:\,[-\epsilon,\epsilon]\to\R$
be defined by $F_\eta(s)={\mathcal F}(\hat K_{\eta,s})$.

Arguing as in the derivation of \eqref{dim 3 F(s)}, we obtain that
\begin{eqnarray*}
F_\eta(s)&=&\int_{\s^2}f(u)\,\Ss(\hat K_{\eta,s},du)\\
&=&\int_{\Omega}f(u)\,\det(\eta I+sQ(\phi,u))\,\H^2(du)
+
\int_{\s^2\setminus\Omega}f(u)\,\Ss(K_\eta,du)\,,
\end{eqnarray*}
where $I$ denotes the identity matrix in ${\mathcal S}_2$ and where we used the fact that 
$\hat h_{\eta,s}$ is of class $C^2$ on $\Omega$. The second integral does not depend on $s$. As
$$
\det(\eta I+sQ(\phi,u))=
\eta^2+s\eta\,{\rm trace}(Q(\phi,u))+
s^2\det(Q(\phi,u))\,,
$$
we get
\begin{equation}\label{derprima}
F_\eta'(0)=\eta\int_{\Omega}
{f(u)\,{\rm trace}(Q(\phi,u))\,
\H^2(du)}
\end{equation}
and
\begin{equation}\label{derseconda}
F_\eta''(0)=2\,\int_{\Omega}
{f(u)\,\det\left(Q(\phi,u)\right)\,\H^2(du)}
=2\,\int_{\mathbb{S}^2}
{f(u)\,\det\left(Q(\phi,u)\right)\,\H^2(du)}.
\end{equation}
Moreover, since $K_\eta$ tends to $C(\overline P,\overline\theta)$ in the Hausdorff distance as
$\eta\to0^+$ and $\mathcal F$ is continuous,
\begin{equation}\label{valfunzione}
F_\eta(0)={\mathcal F}(K_\eta)\rightarrow{\mathcal F}(C(\overline P,\overline\theta))>0
\quad\mbox{as $\eta\to0^+$.}
\end{equation}
In particular, we thus see that $F_\eta$ is twice differentiable and $F_\eta(0)$ is bounded from 
below by a positive constant independent of $\eta$  if $\eta>0$ is sufficiently small.
By Lemma \ref{concavity F in lambda} we know that
$\sqrt {F_\eta}$ is concave, and therefore
\begin{equation}\label{2FF''-F'<0}
2F_\eta(0) F_\eta''(0)-\left(F_\eta'(0)\right)^2\leq 0\,.
\end{equation}
Now the assertion  follows by plugging (\ref{derprima}), (\ref{derseconda}) and (\ref{valfunzione}) into
(\ref{2FF''-F'<0}) and letting $\eta\to0^+$.
\end{proof}

The proof of the following result is implicit in the argument for Lemma \ref{techlemma3} and thus is based
on the regularization argument
contained in Lemma \ref{l:regularization}.

\begin{lema}
\label{lmIV.2} Let $f\in\c(\s^2)$, let $P\in \s^2$ and $\theta\in(0,\pi/2)$.
Assume that
\begin{equation}\label{condizfinale}
\int_{\s^2}{f(u) \det\left(Q(\phi,u)\right)\,\H^2(du)}\leq 0
\end{equation}
holds for all $\phi\in C^2(\s^2)$
with  support contained in $I_{\theta}(P)$. Let $\theta'\in(0,\theta)$.
Then there exists a sequence of functions
$(f_k)_{k\in\N}$ in $\c^\infty(\s^2)$, which converges uniformly to $f$
on $\s^2$, such that for all $k\in\N$, (\ref{condizfinale})
holds with $f$ replaced by $f_k$ and for all $\phi\in C^2(\s^2)$
with  support contained in $I_{\theta'}(P)$.
\end{lema}

\subsection{Proof of Theorem \ref{th dim 3}}\label{dim3dss4}
Let the assumptions of Theorem \ref{th dim 3} be fulfilled.
Our aim is to prove that for each $P\in \s^2$ there is some $\theta\in (0,\pi/2)$
such that for all $\phi\in C^2(\s^2)$ with support contained $I_{\theta}(P)$ the inequality
(\ref{condizfinale}) holds. This will prove Theorem \ref{th dim 3}. Indeed, if this is established
and $P\in \s^2$ is given, let $\theta$ be chosen correspondingly. By Lemma \ref{lmIV.2} there is
a sequence of functions
$(f_k)_{k\in\N}$ in $\c^\infty(\s^2)$, which converges uniformly to $f$
on $\s^2$, such that for all $k\in\N$, (\ref{condizfinale})
holds with $f$ replaced by $f_k$ and for all $\phi\in C^2(\s^2)$
with  support contained in $I_{\theta'}(P)$, where (say) $\theta':=\theta/2$. But then Lemma
\ref{propIV.2} implies that $Q(f_k,u)$ is positive semi-definite for all $u\in I_{\theta'}(P)$ and all $k\in\N$.

By Corollary \ref{appendixcor1}, this shows that the 1-homogeneous extension of $f_k$
is convex in the interior of the cone spanned by $I_{\theta'}(P)$, for every $k$.
The same must then be true for the 1-homogeneous
extension of $f$. In particular, we thus conclude that the 1-homogeneous
extension of $f$ is locally convex on $\R^3\setminus\{o\}$. By a classical result
due to Tietze (see \cite[Theorem 2]{SSV} for a more general result), applied to the epigraph of $f$, it follows that 
the 1-homogeneous
extension of $f$ is  convex on every convex subset of $\R^3\setminus\{o\}$. 
But this easily yields the convexity of the 1-homogeneous
extension of $f$, and thus $f$ is the support function of a convex body.

\medskip

\noindent
{\it Case 1: ${\mathcal F}\ge0$ on convex bodies of class $C^2_+$.} We further divide the treatment of this
case into two subcases. Assume first that there exists a convex body $K\in\K^3$ of class $C^2_+$ such that
${\mathcal F}(K)=0$. Let $h$ be the support function of $K$ and let $\phi\in C^2(\s^2)$.
According to Proposition \ref{appendixprop1} there exists some $\epsilon>0$ such that the function
$h_s=h+s\phi$ is the support function of a convex body $K_s$ of class $C^2_+$, for every $s\in
[-\epsilon,\epsilon]$.
Let $F$ be defined on $[-\epsilon,\epsilon]$ by $F(s)=
{\mathcal F}(K_s)$. Then, in particular, $F$ is twice differentiable and $F$ has a minimum at $s=0$, so that
$F'(0)=0$. Moreover, by (\ref{0.3b}), there exists
a sequence $(s_k)_{k\in\N}$, converging to $0$, such that $F(s_k)=0$ for
every $k\in\N$. Indeed, if on the contrary $F(s)>0$ for every $s\in[-\delta,\delta]\setminus\{0\}$,
for some $\delta\in(0,\epsilon]$, then
\begin{eqnarray*}
0&=&F(0)={\mathcal F}(K)={\mathcal F}\left(\frac12 K_{-\delta}+\frac12 K_{-\delta}\right)\\
&\ge&
\min\{{\mathcal F}(K_{-\delta}),{\mathcal F}(K_{\delta})\}=
\min\{F(-\delta),F(\delta)\}>0\,,
\end{eqnarray*}
which is  a contradiction. Then (as $F'(0)=0$)
$$
F''(0)=2\,\lim_{s\to0}\frac{F(s)-F(0)}{s^2}=2\, \lim_{k\to\infty}\frac{F(s_k)-F(0)}{s^2}=0\,.
$$
On the other hand, by (\ref{I.0a}) we have
$$
F(s)=\int_{\s^2}f(u)\det(Q(h,u)+sQ(\phi,u))\,\H^2(du)\,,
$$
whence, differentiating twice, we conclude that
$$
0=F''(0)=2\,\int_{\s^2}f(u)\det(Q(\phi,u))\,\H^2(du)\,,
$$
i.e.~(\ref{condizfinale}) follows with equality for any $\phi\in C^2(\s^2)$. (In fact, this implies that
$f$ is linear, and thus the support function of a point.)

Assume now that ${\mathcal F}(K)>0$ for every $K\in\K^3$ of class $C^2_+$ .
In this case we aim to prove that the Brunn--Minkowski type inequality (\ref{0.3bb}) implies the
stronger inequality (\ref{BM3d}). Then, according to Proposition \ref{prop int f det<0},
we have (\ref{condizfinale}) in the form required. It is sufficient to prove (\ref{BM3d}) when
the involved bodies are of class $C^2_+$;
the general case follows, since convex bodies of class $C^2_+$ are dense in $\K^3$ and
${\mathcal F}$ is continuous. Let $K_0, K_1\in\K^3$ be of class $C^2_+$ and let
$t\in[0,1]$; moreover, define
$$
\bar K_0=\frac{1}{\mathcal F(K_0)^{1/2}}\,K_0\,,\quad \bar
K_1=\frac{1}{\mathcal F(K_1)^{1/2}}\,K_1\,,\quad
$$
and
\[
\bar t=\frac{t\mathcal F(K_1)^{1/2}}{(1-t)\mathcal F(K_0)^{1/2}+t\mathcal
F(K_1)^{1/2}}\,.
\]
If we apply inequality (\ref{0.3bb}) to $\bar K_0$, $\bar K_1$ and $\bar t$,
we get inequality (\ref{BM3d}) for $K_0$, $K_1$ and $t$.

\medskip

\noindent
{\it Case 2: there exists some $K\in\K^3$ of class $C^2_+$ such that ${\mathcal F}(K)<0$.}
As above, let $h$ be the support function of $K$ and let $\phi\in C^2(\s^2)$;
there exists some $\epsilon>0$ such that the function
$h_s=h+s\phi$ is the support function of a convex body $K_s$ for every $s\in
[-\epsilon,\epsilon]$. Let $F$ be defined in $[-\epsilon,\epsilon]$ by $F(s)=
{\mathcal F}(K_s)$. As $F(0)<0$, we may assume that $F(s)<0$ for every $s\in[-\epsilon,\epsilon]$.
Let $s_0,s_1\in[-\epsilon,\epsilon]$ and $t\in[0,1]$,
and define $K_0=K_{s_0}$ and $K_1=K_{s_1}$; moreover, setting
$$
\bar K_0=\frac{1}{\sqrt{{-\mathcal F}(K_0)}}\,K_0\,,\quad
\bar K_1=\frac{1}{\sqrt{{-\mathcal F}(K_1)}}\,K_1\,,
$$
and
\[
\bar t=
\frac{t\sqrt{{-\mathcal F}(K_1)}}
{(1-t)\sqrt{{-\mathcal F}(K_0)}+t\sqrt{{-\mathcal F}(K_1)}}\,,
\]
we get, from (\ref{0.3bb}), that
\begin{equation*}
\sqrt{-{\mathcal F}((1-t)K_0+tK_1)}\leq
(1-t) \sqrt{-{\mathcal F}(K_0)}
+t\sqrt{-{\mathcal F}(K_1)}\,.
\end{equation*}
Hence, $\sqrt{-F}$ is convex and therefore also
 $-F$ is convex in $[-\epsilon,\epsilon]$, so that
$F''(0)\le0$. On the other hand, as above we have
$$
F''(0)=2\,\int_{\s^2}f(u)\det(Q(\phi,u))\,\H^2(du)\,,
$$
and therefore, in this remaining case, we have proved (\ref{condizfinale}) for an arbitrary
function $\phi\in C^2(\s^2)$.

\section{Proof of Theorem \ref{gral statement}}\label{sV}

The theorem will be proved by induction over the dimension $n$.
The proof in the case $n=3$ has already been given in Section \ref{s: 3 dim case}.

For the induction step, we assume that the result has already been proved
in $\R^n$ for some $n\ge 3$. Let $f:\mathbb{S}^n\subset\R^{n+1}\to\R$ be a
continuous function such that the associated functional $\mathcal{F}$
defined as in (\ref{0.1}) satisfies (\ref{0.3b}).

Clearly, $f$ is the support function of a convex body  if  for any
hyperplane
$H\subset\R^{n+1}$ passing through the origin  the restriction  of $f$ to $H$ is the support
function of a convex body. Let  $H$ be such a hyperplane, and let $e\in \mathbb{S}^n$ be orthogonal to $H$.
For a convex body ${K}\subset H$ and $\lambda> 0$, we define
$$
Z({K},\lambda):=K+\{se:0\le s\le \lambda\}\,.
$$
Thus $Z({K},\lambda)$ is an orthogonal cylinder with bases $K$ and $K+\lambda e$  and of height $\lambda$.
For ${K}_0, {K}_1\subset H$, $\lambda> 0$ and $t\in
[0,1]$, it is easy to check that
$$
Z((1-t) {K}_0+t {K}_1,\lambda)=(1-t)Z( {K}_0,\lambda)+tZ( {K}_1,\lambda)\,.
$$

Subsequently, we denote by $\delta_a$ the Dirac measure with unit point mass at $a\in\R^{n+1}$. Further,
for a convex body $K\subset H$, we denote by ${\rm S}^H_{n-1}(K,\cdot)$   the area measure of $ K$ with
respect to $H$ as ambient space. Using this notation, we now describe a suitable
decomposition of the area measure of $Z(K,\lambda)$.

\begin{lema}\label{areameasurecylinder} Let
${K}\subset H$ be a convex body and $\lambda > 0$.  Then
\begin{equation}\label{misuraareacilindri}
{\rm S}_n(Z({K},\lambda),\cdot)=\mathcal{H}^n({K})\cdot \delta_{-e}(\cdot)
+ \mathcal{H}^n({K})\cdot \delta_{e}(\cdot)+ \lambda\cdot
{\rm S}^H_{n-1}( {K},\cdot\cap H)\,.
\end{equation}
\end{lema}

\begin{proof}
Let $K^\circ$ be the relative interior  and  $\partial_r { K}$  the
relative boundary of $K$ with respect to $H$. The disjoint union of
$K^\circ$, $K^\circ+\lambda e$ and the lateral surface $\partial_r { K}+\{se : 0<s<\lambda\}=:\Sigma_L$ covers
$\partial Z( K,\lambda)$ up to a set of
$\H^2$ measure zero. For  $x\in K^\circ$  we have
${\rm nor}(Z( K,\lambda), x)=\{-e\}$, and for $x\in K^\circ +\lambda e$ we have ${\rm nor}(Z( K,\lambda), x)=\{e\}$.
If $x\in\Sigma_L$, then ${\rm nor}(Z( K,\lambda), x)={\rm nor}^H({ K}, x')$,
where $x'$ is the orthogonal projection of $x$ to $H$ and ${\rm nor}^H({ K},x')$ denotes the set of exterior
unit normal vectors of $K$ at $x'$ with respect to
$H$ as ambient space. According to the definition of area measures, this description of the normal
cones easily leads to (\ref{misuraareacilindri}).
\end{proof}

\medskip

We now turn to the induction step.
Let $\bar{f}:=f|_H$ and $ \mathbb{S}^{n-1}_H:=\mathbb{S}^n\cap
H$. Let $K\subset H$ be a convex body and $\lambda>0$. The preceding lemma yields that
$$
\int_{\mathbb{S}^n} f(u)\,
{\rm S}_n(Z( K,\lambda),du)=[f(e)+f(-e)]\mathcal{H}^n({K})+
\lambda\int_{\mathbb{S}^{n-1}_H}\bar{f}(u)\,{\rm S}^H_{n-1}(K,du).
$$
For arbitrary convex bodies ${K}_0,{K}_1\subset H$, $\lambda>0$ and $t\in [0,1]$ we thus obtain

\begin{align*}
&\mathcal{F}((1-t)Z({K}_0,\lambda)+tZ({K}_1,\lambda))
=\mathcal{F}(Z((1-t){K}_0+t{K}_1,\lambda))\\
&\qquad=\mathcal{H}^n((1-t){K}_0+t{K}_1)\cdot[f(e)+f(-e)]\\
&\qquad\qquad\, +
\lambda\cdot \int_{\mathbb{S}^{n-1}_H}\bar{f}(u)\,{\rm S}^H_{n-1}((1-t){K}_0+t{K}_1,du)\\
&\qquad\ge \min\left\{\mathcal{F}(Z({K}_0,\lambda)),\mathcal{F}(Z({K}_1,\lambda))\right\}\\
&\qquad=\min\left\{\mathcal{H}^n({K}_0)\cdot[f(e)+f(-e)]+
\lambda\cdot \int_{\mathbb{S}^{n-1}_H}\bar{f}(u)\,{\rm S}^H_{n-1}({K}_0,du), \right.\\
&\qquad\qquad\qquad\qquad\left.\mathcal{H}^n({K}_1)\cdot[f(e)+f(-e)]+
\lambda\cdot \int_{\mathbb{S}^{n-1}_H}\bar{f}(u)\,
{\rm S}^H_{n-1}({K}_1,du)\right\}.
\end{align*}
If we divide by $\lambda$ and let $\lambda\to \infty$, we deduce that
\begin{align*}
&\int_{\mathbb{S}^{n-1}_H}\bar{f}(u)\,{\rm S}^H_{n-1}((1-t){K}_0+t{K}_1,du)\\
&\qquad\ge\min\left\{\int_{\mathbb{S}^{n-1}_H}\bar{f}(u)\,{\rm S}^H_{n-1}({K}_0,du)\,,\,
\int_{\mathbb{S}^{n-1}_H}\bar{f}(u)\,{\rm S}^H_{n-1}({K}_1,du)
\right\}\,.
\end{align*}
Since ${K}_0,{K}_1\subset H$ can be chosen
arbitrarily, the inductive hypothesis can be applied to the functional
defined on convex bodies contained in $H$, generated by the function $\bar f$,
and this yields that $\bar{f}$ is convex
in $\mathbb{S}^{n-1}_H$.

\appendix

\section{}

\subsection{Mollification}\label{s: reg arg}
We recall a standard method to approximate continuous functions on the unit sphere
by smooth functions.

Let $\xi:\R\to[0,\infty)$ be a function of class $C^\infty$ with ${\rm
sprt}(\xi)\subseteq[-1,1]$ and $\xi(0)>0$. Then, for $k\in\N$, we define
$\omega_k:\mathbf{O}(n)\to [0,\infty)$ by
$$
\omega_k(\rho):=c_k\cdot \xi(k^2\cdot \|\rho-{\rm id}\|^2)\,,
$$
where $\mathbf{O}(n)$ is the group of rotations of $\R^n$, endowed with
the Haar probability measure $\nu$, ``${\rm id}$'' is the identity element in
${\mathbf{O}(n)}$ and $c_k$ is chosen such that
$$
\int_{\mathbf{O}(n)}\omega_k(\rho)\, \nu(d\rho)=1.
$$
As a composition of $C^\infty$ maps, $\omega_k$ is of class $C^\infty$.
The following lemma is standard.
\begin{lema}
\label{l:regularization} Let $f\in\c(\s^{n-1})$. Then, for $k\in\N$, the
function $f_k:\s^{n-1}\to\R$ defined by
$$
f_k(u):=\int_{\mathbf{O}(n)}f(\rho u)\, \omega_k(\rho)\, \nu(d \rho)\,,\qquad
u\in \s^{n-1}\,,
$$
is of class $C^\infty (\sfe)$, and the sequence $(f_k)_{k\in\N}$ converges
to $f$ uniformly on $\s^{n-1}$.
\end{lema}

\medskip

\subsection{Covariant and Euclidean derivatives}
Let $\Omega$ be an open subset of $\sfe$. Then
the cone generated by the spherical set $\Omega$
is defined by
$ \hat\Omega:=\{tx\,:\,x\in\Omega\,,\, t>0\}$.
For $h\in C^2(\Omega)$, the 1-homogeneous extension $H$ of $h$ is
\begin{eqnarray*}
H(x)=\|x\|\,h\left(\dfrac{x}{\|x\|}\right)\,,\quad x\in\hat\Omega\,;
\end{eqnarray*}
in particular, we have $H\in C^2(\hat\Omega)$.
The next lemma allows us to express, for $u\in\sfe$, the eigenvalues of $Q(h,u)$
in terms of the eigenvalues
of the Hessian matrix $D^2H(u)$ of $H$ in $\R^n$.
The  $(n-1)\times (n-1)$
matrix $Q(h,u)$ involves second
covariant derivatives with respect to a local orthonormal frame of vector fields on $\sfe$. As remarked
earlier, the eigenvalues of  $Q(h,u)$ are independent of the choice of such a frame.
On the other hand, the Hessian matrix $D^2H(u)$ is an $n\times n$ matrix
whose entries  are the (Euclidean) second  partial derivatives $\partial_i\partial_j H(u)$, $i,j=1,\ldots,n$, of $H$
at $u$, determined with respect to a fixed orthonormal system $e_1,\ldots,e_n$ of $\R^n$. Although the Hessian matrix
depends on the choice of this basis, the $n$ eigenvalues of $D^2H(u)$ are independent of such a choice.
A discussion related to the following lemma is contained in \cite[\S 2.5, Lemma 2.5.1] {Schneider} and \cite[\S 3]{Howard}.

\begin{lema}\label{appendixlemma1}
Let $u_0\in\sfe$. If $\lambda_1,\dots,\lambda_{n-1}$ are the
eigenvalues of the matrix $Q(h,u_0)$, then the eigenvalues of the  matrix
$D^2H(u_0)$ are given by $\lambda_1,\dots,\lambda_{n-1},0$.
\end{lema}

\begin{proof} We choose a coordinate system such that $u_0=(0,\dots,0,1)$.
Our first observation is that, by homogeneity, $u_0$ is an eigenvector of
$D^2H(u_0)$, with corresponding eigenvalue $0$. Hence it will be sufficient
to prove that
$$
\partial_i\partial_j H(u_0)
=h_{ij}(u_0)+h(u_0)\delta_{ij}\,,\quad\,
i,j=1,\dots,n-1\,.
$$
On the left-hand side, we consider the (Euclidean) second partial derivatives with respect
to an orthonormal basis $e_1,\ldots,e_n$ with $e_n=u_0$, on the right-hand side,
we consider the covariant derivatives with respect to a local orthonormal frame
which equals $\{e_1,\ldots,e_{n-1}\}$ at $u_0$.
We will write a point $x\in\R^n$ in the form $x=(x',y)$ with $x'\in\R^{n-1}$ and $y\in\R$. Clearly, in this notation
$\|x'\|$ denotes the norm in $\R^{n-1}$. Let
$$
D=\{x'\in\R^{n-1}\,:\,\|x'\|\le1\}\,,
$$
and define the function
$$
\hat h\,:\,D\to\R\,,\quad \hat h(x')=h(x',\sqrt{1-\|x'\|^2})\,.
$$
The second covariant derivatives of $h$ at $u_0$ can be computed through the (Euclidean) second
partial derivatives of $\hat h$ at $o$, that is
\begin{equation}\label{appendix1}
h_{ij}(u_0)=\partial_i\partial_j\hat h(o)\,,\quad\, i,j=1,\dots,n-1\,,
\end{equation}
since we are using  normal coordinates at $u$. 
On the other hand, by the definition of $H$ we have
\begin{eqnarray*}
H(x',1)&=&\sqrt{1+\|x'\|^2}\,h\left(
\frac{x'}{\sqrt{1+\|x'\|^2}},
\frac{1}{\sqrt{1+\|x'\|^2}}
\right)\\
&=&\sqrt{1+\|x'\|^2}\,\hat h\left(
\frac{x'}{\sqrt{1+\|x'\|^2}}
\right)\,,
\quad x'\in \R^{n-1}\,.
\end{eqnarray*}
Hence, for $i,j=1,\dots,n-1$,
$$
\partial_i\partial_j H(u_0)=
\frac{\partial^2}{\partial x'_i\partial x'_j}
\left.
\left(
\sqrt{1+\|x'\|^2}\,\hat h\left(
\frac{x'}{\sqrt{1+\|x'\|^2}}
\right)
\right)
\right|_{x'=o}\,.
$$
The proof is completed by an explicit computation of the derivative on the right hand--side of
the last equality, and by using (\ref{appendix1}).
\end{proof}

The next two results follow from Lemma \ref{appendixlemma1} and the homogeneity of $H$.

\begin{cor}\label{appendixcor1}
In the notation of Lemma \ref{appendixlemma1}, assume moreover that $\hat\Omega$ is convex. Then
$H$ is convex in $\hat\Omega$ if and only if the matrix $Q(h,u)$ is
positive semi-definite for every $u\in\Omega$.
\end{cor}

\begin{cor}\label{appendixcor2}
In the notation of Lemma \ref{appendixlemma1}, assume moreover that $\hat\Omega$ is connected.
If the matrix $Q(h,u)$ is the zero matrix for
every $u\in\Omega$, then $H$ is linear on $\hat\Omega$.
\end{cor}

\medskip

The following result allows us to build a family of perturbations of a convex body having a portion
of the boundary of class $C^2$ with positive Gauss curvature.

\begin{prop}\label{appendixprop1} Let $K\in\K^n$ and let $h$ be its support function.
Assume that there exists an open subset $\Omega$ of $\sfe$ such that $h\in C^2(\Omega)$ and
$Q(h,u)>0$ for every $u\in\Omega$. Let $\phi\in C^2(\sfe)$ and assume that the support of $\phi$
is contained in $\Omega$. Then there exists some $\epsilon>0$ such that the function $h_s:=h+s\phi$
is the support function of a convex body, for every $s\in[-\epsilon,\epsilon]$. In particular, if $\Omega=\sfe$,
then $h_s$ is the support function of a convex body of class $C^2_+$ if $\epsilon>0$ is sufficiently small.
\end{prop}

\begin{proof} Let $D$ be the support of $\phi$, hence $D$ is a compact subset of $\Omega$. As
$Q(h,u)>0$ for every $u\in D$, by compactness there exists $\gamma>0$ such that
$Q(h,u)\ge\gamma I_{n-1}$, where $I_{n-1}$ is the identity matrix in ${\mathcal S}_{n-1}$. Hence
there exists $\epsilon>0$ such that $Q(h_s,u)>0$ for every $u$ in $D$, and consequently in $\Omega$,
and for every $s\in[-\epsilon, \epsilon]$.

Let $H$ and $H_s$ be the 1-homogeneous extensions of $h$ and $h_s$ respectively. We know that $H$ is convex and we
want to prove that $H_s$ is convex as well, for $s\in[-\epsilon, \epsilon]$. Let $x\in\R^n$, $x\ne0$, and let
$u=\frac{x}{\|x\|}$. Assume that $u\in\Omega$ and let ${\mathcal U}$ be a neighborhood of $u$ contained in
$\Omega$ so that $\hat{\mathcal{U}}$ is convex. Then $h_s\in C^2({\mathcal U})$ and, by the previous part of the proof, $Q(h_s,u)>0$ for every
$u\in{\mathcal U}$. Consequently, by Corollary \ref{appendixcor1}, $H_s$ is convex in $\hat{\mathcal U}$ and in
particular it is convex in a neighborhood of $x$. Assume now that $u\notin\Omega$; then there exists
a neighborhood ${\mathcal U}$ of $u$ contained in $\sfe\setminus D$. Then $h_s=h$ in $\mathcal U$ and
consequently $H_s=H$ in $\hat{\mathcal U}$. This proves that $H_s$ is convex in a neighborhood of $x$.
Thus we have shown that $H_s$ is locally convex in $\R^n\setminus\{o\}$. By
a classical result due to Tietze (see \cite{SSV}), applied to the epigraph of $H_s$, this yields that $H_s$ is convex in every
convex subset of $\R^n\setminus\{o\}$. From this the convexity in $\R^n$ is easily obtained.
\end{proof}

\medskip

Let $f\in C^2(A)$, where $A$ is an open subset of $\R^n$ and let $\rho\in{\bf O}(n)$. We denote by
$f_\rho$ the function defined on $\rho(A)$ as the composition of $f$ and $\rho^{-1}$, that is
$$
f_\rho\,:\,\rho(A)\to\R\,,\quad f_\rho(x)=f(\rho^{-1}(x))\,.
$$
Obviously, $f_\rho\in C^2(\rho(A))$.
The chain rule and elementary linear algebra show that $D^2f_\rho (x)$ and $D^2f (\rho^{-1} (x))$
are described by similar (symmetric) matrices with respect to a fixed orthonormal basis, and
hence they have the
same real eigenvalues at $x\in\rho(A)$. Therefore, in particular, we have
$$
\text{trace}( D^2 f_\rho)(x)=\text{trace}( D^2 f)(\rho^{-1}(x))\,,\quad \,x\in\rho(A)\,,
$$
and
$$
\det(D^2 f_\rho(x))=\det(D^2 f(\rho^{-1}(x)))\,,\quad\,x\in\rho(A)\,.
$$
The following result provides similar relations for functions on the sphere.
Let
$\Omega $ be an open subset of $\sfe$, let $\psi\in C^2(\Omega)$, and, for $\rho\in{\bf O}(n)$,
denote by $\psi_\rho$ the function defined on $\rho(\Omega)$
by
$$
\psi_\rho\,:\,\rho(\Omega)\to\R\,,\quad\psi_\rho(x)=\psi(\rho^{-1}(x))\,.
$$

\begin{lema}\label{appendixlemma2} Using the preceding notation, we have
\begin{equation}\label{appendix5}
{\rm trace}(Q(\psi_\rho,x))=
{\rm trace}(Q(\psi,\rho^{-1}(x)))\,,
\quad\,x\in\rho(\Omega)\,;
\end{equation}
\begin{equation}\label{appendix6}
\det(Q(\psi_\rho,x)=\det(Q(\psi, \rho^{-1}(x)))\,,\quad\,x\in\rho(\Omega)\,.
\end{equation}
\end{lema}

\begin{proof} Let $\hat\Omega$ be the cone generated by $\Omega$, and let $\Psi$ be the
$1$-homogeneous extension of $\psi$ to $\hat\Omega$. Clearly, $\Psi_\rho=\Psi\circ\rho^{-1}$ is equal to the
$1$-homogeneous extension of $\psi_\rho$. Let $0,r_1,\ldots,r_{n-1}$ denote the common eigenvalues of
$D^2\Psi_\rho(x)$ and $D^2\Psi(\rho^{-1}(x))$, where $x\in\rho(\Omega)$. By Lemma \ref{appendixlemma1}, applied to $\Psi_\rho$
as the 1-homogeneous extension of $\psi_\rho$ at $x$,
it follows that $Q(\psi_\rho,x)$ has the eigenvalues $r_1,\ldots,r_{n-1}$. In the same way Lemma \ref{appendixlemma1}, applied to $\Psi$ as the 1-homogeneous extension of $\psi$ at $\rho^{-1}(x)$,
shows that $Q(\psi,\rho^{-1}(x))$ has the eigenvalues $r_1,\ldots,r_{n-1}$. Now
 (\ref{appendix5}) and  \eqref{appendix6} follow immediately.
\end{proof}

\end{document}